\documentclass[10pt]{amsart}
\usepackage{amsmath}
\usepackage{xcolor}
\usepackage{tikz}
\usepackage{setspace}
\usepackage{tikz-cd}
\usepackage[all]{xy}
\usepackage{caption}

\usetikzlibrary{arrows,decorations.pathmorphing,backgrounds,positioning,fit}
\newtheorem{theorem}{Theorem}[section]
\newtheorem{lemma}{Lemma}[section]
\newtheorem{defn}{Definition}[section]
\newtheorem{ex}{Example}[section]
\newtheorem{remark}{Remark}[section]
\newtheorem{proposition}[theorem]{Proposition}
\newtheorem{corollary}{Corollary}[section]

\newcommand{\CC}{\mathbb{C}}

\newcommand{\RR}{\mathbb{R}}
\newcommand{\NN}{\mathbb{N}}

\newcommand{\sgn}{\mathrm{sgn}}
\newcommand{\tq}{\tilde{Q}}

\definecolor{jade}{rgb}{0.10, 0.56, 0.42}
\definecolor{cerise}{rgb}{0.87, 0.19, 0.39}

\tikzstyle{mutable}=[inner sep=0.5mm,circle,draw,minimum size=2mm]
\tikzstyle{frozen}=[inner sep=1mm,rectangle,draw]
\tikzstyle{outline}=[thick,line width=1.5mm,draw=black!10]

\title[Rigid potentials for cluster algebras from double Bruhat cells]
{Rigid potentials for cluster algebras from double Bruhat cells}

\author{Maitreyee C. Kulkarni}
\address{School of Mathematics,
Institute for Advanced Study, Princeton}
\email{mkulka2@ias.edu}

\thanks{The author was supported by the NSF Grant No. DMS-1601862 while at LSU and the Charles Simonyi Endowment while at the IAS}

\begin{document}
\begin{abstract} 
Buan, Iyama, Reiten and Smith proved that the superpotential of a quiver corresponding to an element of Coxeter group is rigid. In this paper, we extend this result to the Berenstein-Fomin-Zelevinsky quivers corresponding to double Bruhat cells of symmetric Kac-Moody algebras.  We also realize these quivers as dimers models on cylinder over corresponding graphs.

\end{abstract}
\maketitle

\section{Introduction} 

The categorical models for cluster algebras have been studied since cluster categories were defined in [BMRRT].  On the other hand, \cite{gls} have used modules over preprojective algebras of Dynkin quivers for modular model of cluster algebras.  The triangulated Calabi-- Yau categories (of dimension 2) include the cluster categories and the stable categories of modules over preprojective algebras. The special objects in these categories called cluster tilting objects correspond to clusters in cluster algebras.  Hence the study of these objects is important to understand the corresponding cluster algebras. The endomorphism algebras of the cluster tilting objects are called cluster tilting algebras.  Tilting modules also play an important role in the representation theory of finite dimensional algebras, and the tilted algebras form a central class of algebras.
The relationship between cluster tilting objects and their quivers was studied in \cite{birs11}.

One of the key properties of the cluster tilting objects in their application to cluster algebras is having no loop or 2-cycles in the quivers of the corresponding cluster tilting algebras.  This property is not satisfied by all cluster tilting objects in general.   It is known that the quivers defined in \cite{bfz05} corresponding to a single Weyl group element are associated to the endomorphism algebras of cluster tilting objects in a subcategory of modules over certain quotient of a preprojective algebra.  The cluster-tilted algebras are determined by their quiver.  In particular, any cluster-tilted algebra is isomorphic to a Jacobian algebra of a rigid QP (quiver with potential).  

For a quiver $Q$ with no oriented cycles or loops, let $\tq$ be its double quiver constructed by adding a reverse arrow $a^*$ to each arrow $a$ of $Q$.  Let $c$ be the element defined as:
\[c=\sum_{a \in Q_1} (a^*a- aa^*)\] where $Q_1$ is the set of arrows of $Q$. Consider the ideal $(c)$ generated by the element $c$, then the preprojective algebra $\Lambda$ is 
\[ \Lambda = \CC(\tq) / (c)\]
Let $w$ be an element in the Coxeter group $W_Q$ associated to $Q$. Let $w=s_{i_1}\ldots s_{i_m}$ be its reduced expression.  For each $w$, associate a quotient of $\Lambda$, denoted by $\Lambda_w$, which is independent of the reduced expression of $w$.  Consider the category of the modules over $\Lambda_w$. Let  Sub$\Lambda_w$ be the full subcategory of this category whose objects are the submodules of finitely generated projective $\Lambda_w$-modules.  Then the quiver for the endomorphism algebras of the cluster tilting objects of the stable category $\underline{\text{Sub}}\Lambda_w$ are  given by the Berenstein-Fomin-Zelevinsky quivers from \cite{bfz05}.

These cluster tilted algebras form a large class of 2-CY-tilted algebras. They are Jacobian algebras coming from rigid QPs of BFZ quivers. It is shown that the superpotential of such a quiver is rigid, independently in \cite{birs11}, \cite{k17}.  The BFZ quivers are defined for a pair of Weyl group elements.  The aim of this paper is to prove that the superpotential of all BFZ quivers $Q^{u,v}$ coming from a non-cyclic graph is rigid, using techniques developed in \cite{k17}. A QP that is mutational equivalent to a rigid QP is also rigid, hence this gives rise to many examples of rigid QPs.

In the case of Grassmannians,  Scott showed that the homogeneous coordinate ring of $\mathrm{Gr(k,n)}$ is a cluster algebra \cite{s06} using Postnikov diagrams.  These were defined by Postnikov \cite{p06} and they encode information about seeds of the cluster algebra and its clusters.  Baur, King and Marsh used Postnikov diagrams to give a combinatorial model for this categorification. The authors showed that the completion of the endomorphism algebra is a frozen Jacobian algebra \cite{bkm16}.  Every region in a Postnikov diagram corresponds to a $k$-subset of $\{1, 2, \ldots, n\}$. For a $k$-subset $I$ of $\{1, 2, \ldots, n\}$ corresponding to a minor of the matrix, assign a Cohen--Macaulay $B$-module $M_I$. Then a Postnikov diagram $D$ can be associated with the module $T_D=\bigoplus_I M_I$.  They define a dimer algebra as the Jacobian algebra for the quiver corresponding to a Postnikov diagram. One of the main results in their paper is that the dimer algebra $A_D$ is isomorphic to $\mathrm{End}_B(T_D)$.  This result gives a combinatorial construction of the endomorphism algebra required for their categorification \cite{bkm16}.

 In the previous paper \cite{k17} the author introduced conceptual framework for dimer models for types other than type A.  In this paper, we will see that the BFZ quivers can be realized as dimer models on the cylinder over the graph of $G$. For a graph $\Gamma$ (in particular any Dynkin diagram), let $\Gamma \times \RR^{\geq 0}$ be the cylinder over $\Gamma$ [see ~Figure 1].  A vertex in a graph is called a branch point if it has more than two edges incident to it.  A vertex is called an endpoint if it has exactly one edge incident to it.  Let $V$ be the set of endpoints and branching points of graph.  Let $a$ and $b$ be two vertices in $V$. A path $\Gamma_{a,b}$ between any two vertices  will be called a branch in the graph and $\Gamma_{m,n}\times \RR^{\geq0} $  in the cylinder will be called the sheet of the cylinder.  Any quiver on a cylinder over a Dynkin diagram that satisfying the properties given below will be called a dimer model on the cylinder [see Definition \ref{def:dimer}]:
\begin{itemize}
\item Each face of $Q^{u,v}$ is oriented.
\item Each face of $Q^{u,v}$ on the cylinder $\Gamma \times \RR^{\geq0}$ projects onto an edge of the graph.
\item Each edge of $Q^{u,v}$ projects onto a vertex of the graph or an edge of the graph. 
\end{itemize}

\begin{figure}[ht!]
\begin{tikzpicture}

\node [fill,circle,scale=0.5,label=below:$$] (2) at (2,0) {};
\node [fill,circle,scale=0.5,label=below:$$] (4) at (3.9,-0.7) {};
\node [fill,circle,scale=0.5,label=right:$$] (n) at (3.7,0.5) {};
\node [fill,circle,scale=0.5,label=below:$$] (3) at (3,0) {};
\draw [thick,-,red] (3)-- (4);
\draw [thick,-,cyan] (3)-- (n);
\draw [thick,-,jade] (3)--  (2);
\draw (2,0)-- (2,3);
\draw (3,0)-- (3,3);
\draw [cyan](3.7,0.5)-- (3.7,3.5);
\draw [thick,-,red] (3.9,-0.7)-- (3.9,2.3);
\fill[nearly transparent, green]  (2.93,2.92) rectangle (2.07,0.07);
\fill[nearly transparent,cyan] (3.07,0.12) -- (3.07,2.93)--(3.63,3.43)--(3.63,0.57);
\fill[nearly transparent,red] (3.07,0.02) -- (3.07,2.83)--(3.83,2.19)--(3.83,-0.59);
\end{tikzpicture}

\begin{tikzpicture}
\node [fill,circle,scale=0.5,label=below:$4$] (4) at (4.7,-0.5) {};
\node [fill,circle,scale=0.5,label=right:$2$] (n) at (4.7,0.5) {};
\node [fill,circle,scale=0.5,label=below:$1$] (3) at (3,0) {};
\node [fill,circle,scale=0.5,label=below:$3$] (5) at (4,0) {};
\draw [thick,-] (5)-- (4);
\draw [thick,-] (5)-- (n);
\draw [thick,-] (3)-- (5);
\end{tikzpicture}
\caption{The cylinder over Dynkin diagram of type $D_4$}
\end{figure}
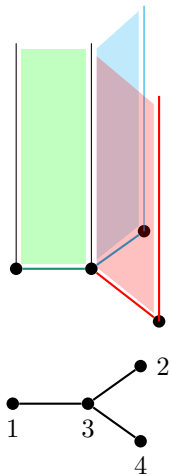

Using this structure on the quivers and the results from \cite{k17}, we prove that BFZ quivers with their superpotentials are rigid QPs. The dimer model structure of these quivers reduces the global problem of verification of rigidity of the super-potential to a local problem on each sheet which allows us to think of a subquiver on a sheet as a dimer model on a disk. In Section 2, we define quivers from double Bruhat cells.  We will also see a detailed example of a quiver $Q^{u,v}$ as defined in \cite{bfz05}.  Section 3 contains some properties of these quivers. In Section 4, we prove our main theorem. First we will see the construction of the quiver $Q^{u,v}$ using the quivers corresponding to a Bruhat cell, i.e. $Q^{u,e}$ and  $Q^{e,v}$.  Then we will prove the following theorem:
\begin{theorem} Let $\Gamma$ be a graph with no loops and cycles. For any Weyl group elements $u,v$ in the corresponding Weyl group, the superpotential of the BFZ quiver $Q^{u,v}$ is rigid. 
\end{theorem}

\section{Quivers from double Bruhat cells} 

Berenstein, Fomin and Zelevinsky defined certain quivers from double Bruhat cells.  In the paper  \cite{bfz05}, they showed that the coordinate rings of double Bruhat cells can be identified with certain upper cluster algebras.  These algebras are defined using the combinatorial data that comes from the corresponding double Bruhat cells. This combinatorial data can be encoded in a quiver. We recall the definition of this quiver in this section.

Let $G$ be a simply connected, connected, semisimple complex algebraic group of rank $r$. Let $B$ and $B_{-}$ be opposite borel subgroups and $W$ be its Weyl group. Then $G$ can be written as \[G=\cup_{u \in W} BuB = \cup_{v \in W} B_{-}vB_{-}.\] The double Bruhat cell $G_{u,v}$ is defined as the intersection $BuB\cap B_{-}vB_{-}$ where $u,v \in W$.  To each such pair of Weyl elements $(u, v) \in W \times W$, we can associate a quiver $Q^{u,v}$ as defined in \cite{bfz05}.  There are two ways to define this quiver.  We can get the quiver directly from the data given by the words $u$ and $v$, or we can define a matrix that gives the quiver.  We will see both ways in this section. 

Let $G$ be a simply connected complex algebraic group. Let $W$ be the Weyl group and $\mathfrak{g}$ be the Lie algebra of $G$. Every Weyl group can be realized as a Coxeter group with reflections ${s_1, s_2, \ldots, s_r}$ of simple roots as its generators. Each $s_i$ is an involution and $(s_is_j)^{m_{ij}} = 1 \in W$, for some integer $m_{ij}$ encoded in the Dynkin diagram.  Every element $w \in W$ has a smallest expression in terms of $s_i$'s.  A word is a tuple of indices of simple reflections in the smallest expression for $w$.   If for $w=s_{i_1}s_{i_2}\cdots s_{i_l}$ in $W$ is the smallest such expression in terms of the generators of $W$ then the word $i =(i_1, i_2, \cdots, i_l), \  i_j\in[1, \cdots, r]$ is said to be in its reduced form.  The length of the word $w$ is denoted by $\ell(w)$ and in this case $\ell(w)= l$.

Fix a pair $(u,v) \in W \times W$. Let us use negative indices for the generators of the first copy of $W$ and positive indices for the second copy of $W$.  Then a reduced word $\textbf{i} = (i_1, \ldots, i_{\ell(u)+\ell(v)})$ is an arbitrary shuffle of a reduced word for $u$ and a reduced word for $v$.  We add the numbers $(-r, \ldots, -1)$ to the tuple $\textbf{i}$ to get a new tuple \[\hat{\textbf{i}}=(-r, \ldots, -1, i_1, \ldots,  i_{\ell(u)+\ell(v)}).\] 
For $k \in [-r, -1] \cup [1, \ell(u)+\ell(v)]$, we define $k^+$ to be the smallest index $l$ such that $k<l$ and $|i_k|=|i_l|$.  If $|i_k|\neq |i_l|$ for any $l>k$, then $k^+= \ell(u)+\ell(v)+1$. An index is called $i$-exchangeable if both $k$ and $k^+$ are in $[-r, -1] \cup [1, \ell(u)+\ell(v)]$.  The set if $i$- exchangeable indices is denoted by $\textbf{e}(\textbf{i})$.

\begin{defn}\label{def:bfz} Let $u, v \in W$.  A BFZ quiver $Q^{u,v}$ has set of vertices $Q_0 = \hat{\textbf{i}}$  Vertices $k$ and $l$ such that $k < l$ are connected if and only if either $k$ or $l$ are $i$-exchangeable. There are two types of edges: 
\begin{itemize}
\item  An edge is called horizontal if $l = k^+$ and it is directed from $k$ to $l$ if and only if $\epsilon(i_l) =+1$.
\item An edge is called inclined if one of the following conditions hold: 
\begin{enumerate}
\item $l < k^+ < l^+$, $a_{|i_k|, |i_j|} < 0$, $\epsilon(i_l) = \epsilon(i_{k^+})$
\item $l < l^+ < k^+$, $a_{|i_k|, |i_j|} < 0$, $\epsilon(i_l) = -\epsilon(i_{l^+})$
\end{enumerate} 
An inclined edge is directed from $k$ to $l$ if and only if $\epsilon(i_l) =-1$.
\end{itemize}
\end{defn}

Let $\widetilde{B}(\textbf{i})$ be the matrix corresponding to this quiver. Its rows are labelled by the set $ [-r, -1] \cup [1, \ell(u)+\ell(v)]$ and columns are labelled by $\textbf{e}(\textbf{i})$.  It is defined as follows:
\[   
b_{kl} = 
     \begin{cases}
       -\sgn(k-l)\epsilon(i_p) &\quad\text{if $p= q$}\\
       -\sgn(k-l)\epsilon(i_p)a_{|i_k|,|i_l|} &\quad\text{if $p < q$ and $(k-l)(k^+-l^+)\epsilon(i_p)\epsilon(i_q) > 0$} \\
        0 &\quad\text{otherwise} \\
       \end{cases}
\]
where $p=\max \{k, l\}$, $q=\min\{k^+, l^+\}$ and $a_{|i_k|,|i_l|}$ is the corresponding entry in the Cartan matrix.
The vertices of the quiver correspond to the set $ [-r, -1] \cup [1, \ell(u)+\ell(v)]$.  The edges are given by the matrix entries.  Two vertuces $k$ and $l$ are connected if and only if $b_{kl} \neq 0$.  If $b_{kl}>0$ then the edge is directed from $k$ to $l$.  If $b_{kl}<0$, then the edge goes from $l$ to $k$. 

We will see an example of a BFZ quiver below.
Consider the group $SL_4(\CC)$.  Here $B$ is the Borel group of upper-triangular matrices and $B_-$ is the group of lower-triangular matrices. The Weyl group in this case is $W=S_4$, the permutation group on four elements. Let $u=w_0=s_3s_2s_1s_2s_3, v=e \in S_4$. So $\ell(u)=5$ and $\ell(v)=0.$  The element $u$ is the longest element of $W$. The quiver $Q^{u,v}$ corresponding to the double Bruhat cell $G_{u,v}$ is as shown below in Figure~\ref{sl4}. (The Dynkin diagram $A_3$ is not a part of the quiver.). Following is the detailed computations for this example.
 
The word $u=w_0=s_3s_2s_1s_2s_3$ is a reduced word as no braid relation can reduce its length. As the second word $v$ is the identity, it does not contribute any vertex or an edge to the quiver $Q^{u,v}$.   So, $\ell(u)+\ell(v)=5+0$, $r=3$ and
\[\hat{\textbf{i}}=(-3, -2, -1, \ 3, \ 2 \ ,1, \ 2, \ 3) \ \ \   \text{or} \] 
\[\begin{array}{c|c|c|c|c|c|c|c|c}
 k & -3 & -2 & -1& 1 & 2 & 3 & 4 & 5\\
 \hline
 \hat{\textbf{i}} & -3 & -2 & -1& 3 & 2 & 1 & 2 & 3 \\
     \hline
i_k & i_{-3} & i_{-2} & i_{-1} & i_{1}  & i_{2}  & i_{3}  & i_{4}  & i_{5} 
 \end{array}\]

Let us compute $k^+$ for each $k$.  From the definition of $k^+$, we know that it tells the next entry in $\hat{\textbf{i}}$ which matches $i_k$ up to sign.  For example, for $k=-3$, $k^+ =1$ because $|i_{-3}|= |i_1| = 3$ and there is no appearance of 3 or $-3$ between those two (i.e. if $k= -3$, $k^+$ cannot be 5 even though $|i_{-3}|= |i_5| = 3$ because those are not the consecutive appearances of 3 or $-3$).  The following table shows $k^+$ for this example. 
\[\begin{array}{c|c|c|c|c|c|c|c|c}
 k & -3 & -2 & -1& 1 & 2 & 3 & 4 & 5\\
 \hline
k^+ & 1 & 2 & 3& 5 & 4 & 6 & 6 & 6 
 \end{array}\]
If $|i_k|\neq |i_l|$ for any $l>k$, then $k^+= \ell(u)+\ell(v)+1$.  For example, for $k = 3$, $i_k=i_3=1$ is the last appearance of 1 in $ \hat{\textbf{i}}$ and so $k^+= 6$. Similarly, $4^+=5^+=6$.  An index is called $i$-exchangeable if both $k$ and $k^+$ are in $[-r, -1] \cup [1, \ell(u)+\ell(v)]$.  So 3, 4, 5 are not $i$-exchangeable.  Also, $-3, -2, -1$ are not $i$-exchangeable. The only $i$-exchangeable indices are $k=1, 2$.  Therefore, $\textbf{e}(\textbf{i})=\{1, 2\}$.  

The matrix $\widetilde{B}(\textbf{i})$ is an $8 \times 2$ matrix whose rows are labelled by (-3, -2, -1, 1, 2, 3, 4, 5) and columns are labelled by (1,2). We will compute the entries $b_{kl}$ in the following table.  If $p=q$, then the entry in the matrix is zero, and so we do not compute the rest of the values in that particular row.  Similarly, if $p >q$, then the corresponding entry in the matrix is zero.  For example, the entry $b_{-3, 2}$ is zero since $p >q$ in that row, so we do not meed to compute the rest of the entries.  When $k=l$, the entry $b_{kl}=b_{kk}=0$ since the sign of $(k-l)$ determines the entry. 

\begin{center}
$\begin{array}{c|c|c|c|c|c|c|c|c|c}
 k & l & p = & q= & \epsilon(i_p) & \epsilon(i_q) & \sgn(k-l) & \sgn(k^+-l^+) &a_{|i_k|,|i_l|}& b_{kl}\\
  &  & \max(k, l) & \min(k^+, l^+) &  &  &  &  &  & \\
 \hline
-3 & 1 & 1 & 1& + &  + & - &  &  &1\\
 \hline
-2 & 1 & 1 & 2& + & + & - & - & -1 &-1\\
 \hline
-1 & 1 & 1 & 3& + & + & - & - & 0 & 0\\
 \hline
1 & 1 & 2 & 5& + & + & 0 &  &  &0\\
 \hline
2 & 1 & 2 & 4& + & + & + & - & -1 & 0\\
 \hline
3 & 1 & 3 & 5& + & + & +& - & 0 & 0\\
 \hline
4 &1 & 4 & 5& + & + & + & - & -1 & 1\\
 \hline
5 & 1 & 5 & 5& + & + & + &  &  &-1\\
 \hline
-3 & 2 & 2 & 1&  &  &  &  &  &0\\
 \hline
-2 & 2 & 2 & 2& + & + & - & - &  & 1\\
 \hline
-1 & 2 & 2 & 3& + & + & - & - & -1 &-1\\
\hline
1 & 2 & 2 & 4& + & + & - & - & -1 &0\\
 \hline
2 & 2 & 2 & 4& + & + & 0 &  &  &0\\
 \hline
3 & 2 & 3 & 4& + & + & +& - & -1 & 1\\
 \hline
4 & 2 & 4 & 4& + & + & + &  &  &-1\\
 \hline
5 & 2 & 5 & 4 &  & & &  &  & 0\\
\hline
 \end{array}$
 \end{center}

 So the matrix $\widetilde{B}(\textbf{i})$ is:
 
 \[\begin{array}{c|ccc}
  & 1& & 2 \\
 \hline
-3 & 1 && 0\\
-2 & -1 && 1 \\
-1 & 0 && -1 \\
1 & 0 & & 0 \\
2 & 0 & & 0 \\
3  & 0 & & 1\\
4  & 1 & & -1\\
5  & -1 & & 0\\
 \end{array}\]
 Since $b_{-3,1} = 1$, the edge in the quiver is directed from -3 to 1.  On the other hand, since $b_{-2,1}=-1$,  the edge is directed from 1 to -2.  There are no edges between $k$ and $l$ if $b_{kl}=0$.  So the quiver $Q^{w_0,e}$ is as shown below:
  \begin{figure}[ht!]
\begin{center}
\begin{tikzpicture}
\node [frozen] (v2) at (0,0) {5};
\node [mutable] (v5) at (0,1.5) {1};
\node [frozen] (vn) at (0,3) {-3};

\node [frozen] (v3) at (1.5,0) {4};
\node [mutable] (v1) at (1.5,1.5) {2};
\node [frozen] (v6) at (1.5,3) {-2};

\node [frozen] (v4) at (3,0) {3};
\node [frozen] (v) at (3,1.5) {-1};
\node [fill,circle,scale=0.5,label=below:3] (1) at (0,-1.5) {};
\node [fill,circle,scale=0.5,label=below:2]  (2)at (1.5,-1.5) {};
\node [fill,circle,scale=0.5,label=below:1] (3) at (3,-1.5) {};
\node at (-1.5,-1.5) {$A_3:$};
\node at (-1.5,1.5) {$Q^{w_0,e}:$};
\draw (1) -- (2);
\draw (2) -- (3);
\draw [thick,->] (v5)-- (v2);
\draw [thick,->] (v1)-- (v3);
\draw [thick,->] (vn)-- (v5);
\draw [thick,->](v6) -- (v1) ;
\draw [thick,->](v) -- (v4) ;
\draw [thick,->] (v3)-- (v5);
\draw [thick,->] (v5)-- (v6);
\draw [thick,->] (v1)-- (v);
\draw [thick,->] (v4)-- (v1);
\end{tikzpicture}
\caption{A BFZ quiver in type $A$}
\label{sl4}
\end{center}
\end{figure}
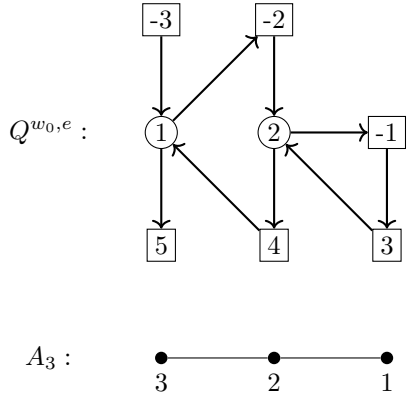
 
 \begin{figure}[ht!]
\begin{center}
\begin{tikzpicture}
\node [frozen] (v2) at (0,0) {5};
\node [mutable] (v5) at (0,1.5) {1};
\node [frozen] (vn) at (0,3) {-3};

\node [frozen] (v3) at (1.5,0) {4};
\node [mutable] (v1) at (1.5,1.5) {2};
\node [frozen] (v6) at (1.5,3) {-2};

\node [frozen] (v4) at (3,0) {3};
\node [frozen] (v) at (3,1.5) {-1};
\node [fill,circle,scale=0.5,label=below:3] (1) at (0,-1.5) {};
\node [fill,circle,scale=0.5,label=below:2]  (2)at (1.5,-1.5) {};
\node [fill,circle,scale=0.5,label=below:1] (3) at (3,-1.5) {};
\node at (-1.5,-1.5) {$A_3:$};
\node at (-1.5,1.5) {$Q^{w_0,e}:$};
\draw (1) -- (2);
\draw (2) -- (3);
\draw [thick,->] (v5)-- (v2);
\draw [thick,->] (v2)-- (v3);
\draw [thick,->] (v1)-- (v3);
\draw [thick,->] (v3) -- (v4) ;
\draw [thick,->] (vn)-- (v5);
\draw [thick,->](v6) -- (v1) ;
\draw [thick,->](v) -- (v4) ;
\draw [thick,->] (v3)-- (v5);
\draw [thick,->] (v5)-- (v6);
\draw [thick,->] (v6)-- (vn);
\draw [thick,->] (v)-- (v6);
\draw [thick,->] (v1)-- (v);
\draw [thick,->] (v4)-- (v1);
\end{tikzpicture}
\caption{A BFZ quiver in type $A$ with extra arrows}
\label{sl4frozen}
\end{center}
\end{figure}
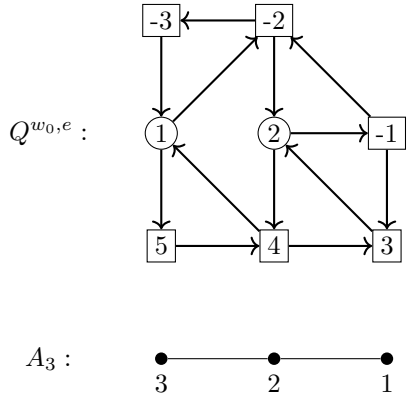

 The circular vertices are mutable and the square vertices are frozen.  The definition above does not include edges between certain frozen vertices.  Moreover edges between frozen vertices are usually not shown as they do not contribute any information to the cluster algebra.  But in this article, we will add the arrows between frozen vertices that complete simple cycles, as shown in~Figure \ref{sl4frozen}.  Note that every cycle in both of these quivers is oriented.  The quiver is drawn in such a way that the number of $s_i$'s in the words $u$ and $v$ correspond to the number of arrows of the quiver that lie directly above the vertex $i$ of the Dynkin diagram.  For example, there are two $s_3$'s in $w_0$ and $e$ together, which correspond to the two vertical arrows in the quiver that lie above the vertex 3 of $A_3$, similarly for $s_2$ and $s_1$.

A quiver for double Bruhat cells for $A_n$ can be viewed as a quiver on a plane on $A_n$ as shown in the figure above. Observe that:
\begin{itemize}
\item we have drawn the quiver such that all vertices lie on a straight line above a vertex of the Dynkin diagram.  Let us call these lines strings;
\item all vertical edges in the quiver project onto vertices in the Dynkin diagram, i.e. all vertical edges lie strictly on the strings;
\item all inclined edges project onto edges of the Dynkin diagram.  In other words, there are no edges that connect two vertices lying on non-adjacent strings.
\end{itemize}

This structure can be generalized to BFZ quivers outside of type $A$. In order to do this, we will define quivers on cylinders over Dynkin diagrams, and then show that the BFZ quivers are examples of those.

\section{Quivers on cylinders over Dynkin diagrams}

Let $\Gamma$ be a graph. A vertex of a graph is called an endpoint if it has only one edge incident to it. A vertex is called a ramification point if it has strictly more than two edges incident to it.  A path $\Gamma_{m, n}$ between two vertices $m$ and $n$ in $\Gamma$ is called a branch if both $m$ and $n$ are branching points or endpoints or if one of them is a branching point and the other is an endpoint. 
\begin{defn} The cylinder over a graph $\Gamma$ is the topological space $\Gamma \times \RR^{\geq 0}$. Let $\Gamma_0$ be the set of vertices of $\Gamma$.  The set $\Gamma_{m,n} \times \RR^{\geq 0}$ is called the sheet over the branch $\Gamma_{m,n}$.  The length of a sheet is the number of edges on the branch.  The subset $\{x\} \times \RR^{ \geq 0}$ where $x \in \Gamma_0$ is called a string.  The string over a ramification point will be called a spine.
\end{defn}

\begin{defn} \label{def:dimer} A quiver on the cylinder over a graph $\Gamma$ is called a dimer quiver on the cylinder if
\begin{enumerate} 
\item Each arrow of the quiver projects onto an edge or a vertex of $\Gamma$.
\item Each face of the quiver projects onto an edge of $\Gamma$.
\item Each face of the quiver is oriented. 
\item The first and last vertices on each string are frozen. The first and last faces on each strip have an edge that connects two frozen vertices.
\item Two faces do not share two edges unless their common string is a spine in the cylinder.
\end{enumerate}
\end{defn}

\begin{ex} 
Quivers for double Bruhat cells of $E_6$ can be drawn on a book-like structure as shown in ~Figure \ref{fig:e7dimer}. The cylinder $E_6 \times \RR^{\geq 0}$ has six strings and three sheets: two sheets of length 2 (green and red) and one sheet of length 1 (blue) glued together at their boundaries (the black string). We will call the black string, a spine of the cylinder.
\end{ex}

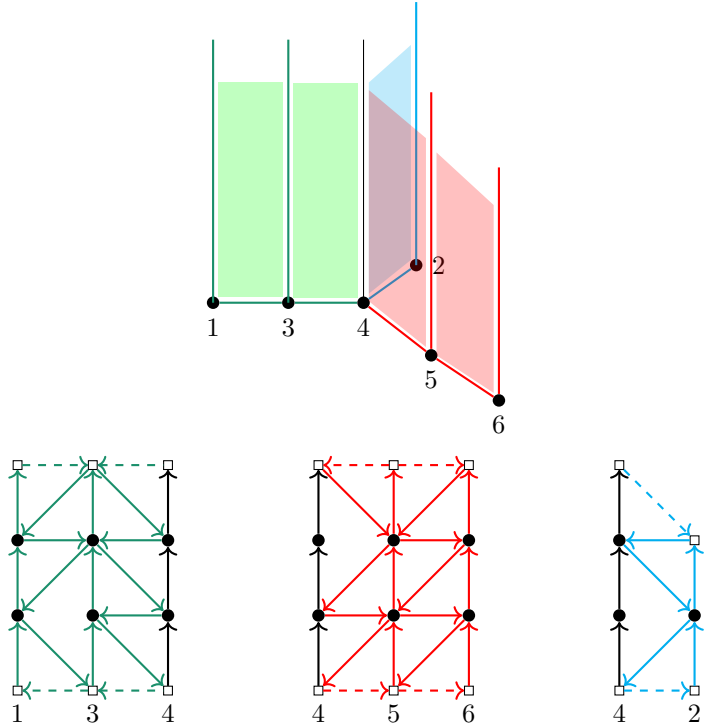
\begin{figure}[ht!]
\begin{center}
\begin{tikzpicture}
\node [fill,circle,scale=0.5,label=below:$1$] (1) at (1,0) {};
\node [fill,circle,scale=0.5,label=below:$3$] (3) at (2,0) {};
\node [fill,circle,scale=0.5,label=below:$5$] (5) at (3.9,-0.7) {};
\node [fill,circle,scale=0.5,label=right:$2$] (2) at (3.7,0.5) {};
\node [fill,circle,scale=0.5,label=below:$4$] (4) at (3,0) {};
\node [fill,circle,scale=0.5,label=below:$6$] (6) at (4.8,-1.3) {};

\draw [thick,-,jade] (1)-- (3);
\draw [thick,-,jade] (3)-- (4);
\draw [thick,-,red] (4)-- (5);
\draw [thick,-,red] (5)-- (6);
\draw [thick,-,red] (6)-- (4.8,1.8);
\draw [thick,-,cyan] (2)-- (4);
\draw  [thick,-,jade] (1,0)-- (1,3.5);
\draw  [thick,-,jade] (2,0)-- (2,3.5);
\draw (3,0)-- (3,3.5);
\draw [thick,cyan](3.7,0.5)-- (3.7,4);
\draw [thick,-,red] (5)--  (3.9,2.8);
\fill[nearly transparent,green]  (2.93,2.92) rectangle (2.07,0.07);
\fill[nearly transparent, green] (1.07,0.07) rectangle (1.93,2.93);
\fill[nearly transparent,cyan] (3.07,0.12) -- (3.07,2.93)--(3.63,3.43)--(3.63,0.57);
\fill[nearly transparent,red] (3.07,0.02) -- (3.07,2.83)--(3.83,2.19)--(3.83,-0.59);
\fill[nearly transparent,red] (3.97,-0.7) -- (3.97,2)--(4.73,1.3)--(4.73,-1.2);
\end{tikzpicture}

\medskip

\begin{tikzpicture}

\node [rectangle,draw,scale=0.5,label=below:$1$] (1) at (1,0) {};
\node [fill,circle,scale=0.5,label=below:$$] (2) at (1,1) {};
\node [fill,circle,scale=0.5,label=below:$$] (3) at (1,2) {};
\node [rectangle,draw,scale=0.5,label=below:$$] (4) at (1,3) {};
\node [rectangle,draw,scale=0.5,label=below:$3$] (5) at (2,0) {};
\node [fill,circle,scale=0.5,label=below:$$] (6) at (2,1) {};
\node [fill,circle,scale=0.5,label=below:$$] (7) at (2,2) {};
\node [rectangle,draw,scale=0.5,label=below:$$] (8) at (2,3) {};
\node [rectangle,draw,scale=0.5,label=below:$4$] (9) at (3,0) {};
\node [fill,circle,scale=0.5,label=below:$$] (10) at (3,1) {};
\node [fill,circle,scale=0.5,label=below:$$] (11) at (3,2) {};
\node [rectangle,draw,scale=0.5,label=below:$$] (12) at (3,3) {};

\draw [thick,->, jade] (1)-- (2);
\draw [thick,->, jade] (2)-- (3);
\draw [thick,->, jade] (3)-- (4);
\draw [thick,->, jade] (5)-- (6);
\draw [thick,->, jade] (6)-- (7);
\draw [thick,->, jade] (7)-- (8);
\draw [thick,->] (9)-- (10);
\draw [thick,->] (10)-- (11);
\draw [thick,->] (11)-- (12);

\draw [thick,->, jade] (2)-- (5);
\draw [thick,<-, jade] (2) --(7);
\draw [thick,->, jade] (3)-- (7);
\draw [thick,->, jade] (8)-- (3);
\draw [thick,->, dashed,jade] (5)-- (1);
\draw [thick,->, dashed,jade] (4)-- (8);

\draw [thick,->, jade] (6)--(9);
\draw [thick,->, jade] (8)--(11);
\draw [thick,->, jade] (10)--(6);
\draw [thick,->, jade] (7)--(10);
\draw [thick,->, jade] (11)--(7);
\draw [thick,<-, dashed,jade] (5)-- (9);
\draw [thick,<-, dashed,jade] (8)-- (12);

\node [rectangle,draw,scale=0.5,label=below:$4$] (13) at (5,0) {};
\node [fill,circle,scale=0.5,label=below:$$] (14) at (5,1) {};
\node [fill,circle,scale=0.5,label=below:$$] (15) at (5,2) {};
\node [rectangle,draw,scale=0.5,label=below:$$] (16) at (5,3) {};
\node [rectangle,draw,scale=0.5,label=below:$5$] (17) at (6,0) {};
\node [fill,circle,scale=0.5,label=below:$$] (18) at (6,1) {};
\node [fill,circle,scale=0.5,label=below:$$] (19) at (6,2) {};
\node [rectangle,draw,scale=0.5,label=below:$$] (20) at (6,3) {};
\node [rectangle,draw,scale=0.5,label=below:$6$] (21) at (7,0) {};
\node [fill,circle,scale=0.5,label=below:$$] (22) at (7,1) {};
\node [fill,circle,scale=0.5,label=below:$$] (23) at (7,2) {};
\node [rectangle,draw,scale=0.5,label=below:$$] (24) at (7,3) {};

\draw [thick,->] (13)-- (14);
\draw [thick,->] (14)-- (15);
\draw [thick,->] (15)--  (16);
\draw [thick,->, red] (18)--  (13);
\draw [thick,->, red] (14)-- (18);
\draw [thick,->, red] (19)-- (14);
\draw [thick,->, red] (16)-- (19);

\draw [thick,->, red] (24)--(19);
\draw [thick,<-, red] (23)--  (19);
\draw [thick,->, red] (18)--(22);
\draw [thick,<-, red] (18)--(23);
\draw [thick,->, red] (22)--(17);
\draw [thick,->, dashed,red] (20)-- (16);
\draw [thick,->, dashed,red] (13)-- (17);

\draw [thick,->, red] (17)-- (18);
\draw [thick,->, red] (18)-- (19);
\draw [thick,->, red] (19)--  (20);
\draw [thick,->, red] (21)--  (22);
\draw [thick,->, red] (22)--  (23);
\draw [thick,->, red] (23)--  (24);
\draw [thick,->, dashed,red] (20)-- (24);
\draw [thick,->, dashed,red] (17)--(21);

\node [rectangle,draw,scale=0.5,label=below:$4$] (25) at (9,0) {};
\node [fill,circle,scale=0.5,label=below:$$] (26) at (9,1) {};
\node [fill,circle,scale=0.5,label=below:$$] (27) at (9,2) {};
\node [rectangle,draw,scale=0.5,label=below:$$] (28) at (9,3) {};
\node [rectangle,draw,scale=0.5,label=below:$2$] (29) at (10,0) {};
\node [fill,circle,scale=0.5,label=below:$$] (30) at (10,1) {};
\node [rectangle,draw,scale=0.5,label=below:$$] (31) at (10,2) {};

\draw [thick,->,cyan] (29)--(30);
\draw [thick,->,cyan] (30)-- (31);
\draw [thick,->] (25)-- (26);
\draw [thick,->] (26)-- (27);
\draw [thick,->] (27)-- (28);
\draw [thick,->, cyan] (30)-- (25);
\draw [thick,->, cyan] (27)-- (30);
\draw [thick,->, cyan] (31)-- (27);

\draw [thick,->, dashed,cyan] (25)-- (29);
\draw [thick,->, dashed,cyan] (28)-- (31);

\end{tikzpicture}
 \caption{A BFZ quiver on the cylinder over $E_6$ corresponding to the Weyl element $s_1s_3s_2s_5s_4s_3s_6s_1s_5s_6s_4s_3s_2s_1s_4s_5s_6$.  The quivers lie on the corresponding colored sheets, and they share the three black arrows on the spine. The square vertices are frozen.}
\label{fig:e7dimer}
\end{center}
\end{figure}

Even though the following theorem is proved for Dynkin diagrams, it is true for any Kac-- Moody algebra and its graph.  Our examples will use Dynkin diagrams and some of the theorems in Section 4 will be proved for trees, particularly for Dynkin diagrams.

\begin{theorem} A BFZ quiver can be realized as a dimer model on the cylinder over the corresponding Dynkin diagram.
\end{theorem}
\begin{proof} Notice that the horizontal edges in ~Definition \ref{def:bfz} lie on the strings of the cylinder over the Dynkin diagram.  All inclined edges lie between two adjacent strings such that they project down onto an edge of the Dynkin diagram.   In other words, there does not exist an edge between strings that are not adjacent.  This is because of the condition $a_{|i_k|, |i_j|} < 0$ on the inclined edges in ~Definition \ref{def:bfz} which assures that there is an edge between two vertices of adjacent strings only if the corresponding vertices in the graph are connected by an edge.  This implies that each face lies between two adjacent strings, satisfying the second condition in ~Definition \ref{def:dimer}.  Each face of the quiver is oriented (this will be shown in Proposition 3.3)  For the fourth condition, notice that the first and the last vertex on each string are frozen according to the definition of the BFZ quivers. 
\end{proof}

\begin{theorem}\label{Theorem:gluing}
For any $u, v \in W$, the quiver $Q^{u,v}$ can be obtained from gluing $Q^{e,v}$ on top of $Q^{u,e}$ in the following way: 
\begin{itemize}
\item On each string of the quiver, identify the bottom frozen vertex of $Q^{e,v}$ to the top frozen vertex of $Q^{u,e}$. 
\item The identified vertices are mutable vertices of $Q^{u,v}$ as they no longer are the boundary vertices. 
\item If the edges between two identified pair of vertices are directed in the same direction, then we keep one edge between them.  If the edges are not in the same direction then we delete the edges, so there is no edge between those vertices in $Q^{u,v}$. 
\end{itemize}
\end{theorem}

\begin{proof} It is enough to prove this for any two neighboring strings in the cylinder.  Let $r$, $p$ be two neighboring vertices in the Dynkin diagram, i.e. $M(r,p)<0$ where $M$ is the Cartan matrix. 
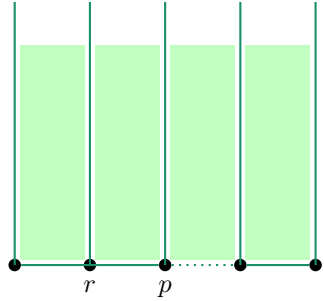
\[ 
\begin{tikzpicture}
\node [fill,circle,scale=0.5,label=below:$$] (1) at (0,0) {};
\node [fill,circle,scale=0.5,label=below:$r$] (2) at (1,0) {};
\node [fill,circle,scale=0.5,label=below:$p$] (3) at (2,0) {};
\node [fill,circle,scale=0.5,label=below:$$] (4) at (3,0) {};
\node [fill,circle,scale=0.5,label=below:$$] (5) at (4,0) {};

\draw [thick,-,jade] (1)-- (3);
\draw [thick,dotted,jade] (3)-- (4);
\draw [thick,-,jade] (4)-- (5);
\draw  [thick,-,jade] (0,0)-- (0,3.5);
\draw  [thick,-,jade] (1,0)-- (1,3.5);
\draw  [thick,-,jade] (2,0)-- (2,3.5);
\draw  [thick,-,jade] (3,0)-- (3,3.5);
\draw  [thick,-,jade] (4,0)-- (4,3.5);

\fill[nearly transparent, green] (0.07,0.07) rectangle (0.93,2.93);
\fill[nearly transparent, green] (1.07,0.07) rectangle (1.93,2.92);
\fill[nearly transparent, green] (2.07,0.07) rectangle (2.93,2.93);
\fill[nearly transparent, green] (3.07,0.07) rectangle (3.93,2.93);

\end{tikzpicture}\]

Let $u$ and $v$ be two reduced Weyl group elements. Depending on the first and last positions of $s_r$ and $s_p$ in the words $u$ and $v$, there are four possible cases:
\begin{itemize}
\item $u= \_\_\_s_r\_\_\_ s_p\_\_\_$, $v=\_\_\_s_r\_\_\_ s_p\_\_\_$
\item $u= \_\_\_s_r\_\_\_ s_p\_\_\_$, $v=\_\_\_s_p\_\_\_ s_r\_\_\_$
\item $u= \_\_\_s_p\_\_\_ s_r\_\_\_$, $v=\_\_\_s_r\_\_\_ s_p\_\_\_$
\item $u= \_\_\_s_p\_\_\_ s_r\_\_\_$, $v=\_\_\_s_p\_\_\_ s_r\_\_\_$
\end{itemize}

Let us consider the first case where 
\[u= \_\_\_\underset{\substack{\uparrow \\ \text{$k_1$th}}}{s_r}\_\_\_ \underset{\substack{\uparrow \\ \text{$k_2$th}}}{s_p}\_\_\_, v= \_\_\_\underset{\substack{\uparrow \\ \text{$k_3$th}}}{s_r}\_\_\_ \underset{\substack{\uparrow \\ \text{$k_4$th}}}{s_p}\_\_\_.\]  
We know that the faces in $Q^{u,e}$ and $Q^{e,v}$ are oriented.  In this case, the vertices $k_1$ and $k_2$ are frozen vertices of $Q^{u,e}$, $l_0$ and $l_1$ are frozen vertices of $Q^{e,v}$.  

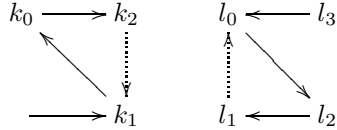
\[ \xymatrix{
 k_0 \ar[r] &  k_2\ar@{.>}[d]  & l_0 \ar[dr] & l_3 \ar[l]\\
\ar[r]  & k_1 \ar[ul] & l_1 \ar@{.>}[u]& l_2 \ar[l]
}\]

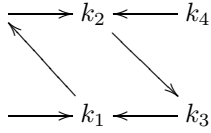
\[ \xymatrix{
 \ar[r] &  k_2 \ar[dr] & k_4 \ar[l]\\
\ar[r]  & k_1 \ar[ul] & k_3 \ar[l]
}\]


\begin{center}
$
\begin{array}{c|c|c}
k & i_k & k^{+} \\
\hline
k_1 & r & k_3    \\
k_2 & p & k_4    \\
k_3 & -r & k_3^{+} > k_4   \\
k_4 & -p & k_4^{+} \geq k_4+1
 \end{array}
 $
\end{center}
 
 As  $k_1 \leq k_2$ and  $\sgn(i_{k_2})\neq \sgn(i_{k_4})$, we need to check for the inequality $k_1< k_2 < k_2^{+} < k_1^{+}$. But the inequality is not true because $k_2^+ = k_4 > k_3 =k_1^+$. So there is no edge between $k_1$ and $k_2$ in $Q^{u,v}$.
 

For the second case where $u= \_\_\_s_r\_\_\_ s_p\_\_\_$, $v=\_\_\_s_p\_\_\_ s_r\_\_\_$: Again, the faces in $Q^{u,e}$ and $Q^{e,v}$ are oriented.  The vertices $k_1$ and $k_2$ are frozen vertices of $Q^{u,e}$, $l_0$ and $l_1$ are frozen vertices of $Q^{e,v}$. 
The following table of $k$ and $k^+$ in $Q^{u,v}$ shows that $k_1< k_2 < k_3 < k_1^{+}$, $sign(i_{k_2}) \neq sign(i_{k_3})$, we also know that $M(|i_{k_1}|, |i_{k_2}|) = M(r,p) < 0$.  Hence there exists an edge between $k_1$ and $k_2$.

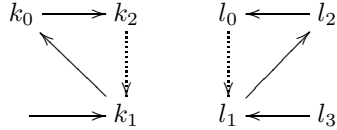
\[ \xymatrix{
k_0 \ar[r] &  k_2 \ar@{.>}[d] & l_0 \ar@{.>}[d] & l_2 \ar[l]\\
\ar[r]  & k_1 \ar[ul] & l_1\ar[ur]  & l_3 \ar[l]
}\]

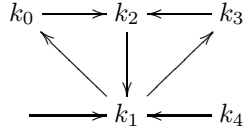
\[ \xymatrix{
k_0 \ar[r] &  k_2 \ar[d] & k_3 \ar[l]\\
\ar[r]  & k_1 \ar[ul] \ar[ur] & k_4 \ar[l]
}\]

\begin{center}
$
\begin{array}{c|c|c}
k & i_k & k^{+} \\
     \hline
k_1 & r & k_4    \\
k_2 & p & k_3    \\
k_3 & -p & k_3^{+} > k_4   \\
k_4 & -r & k_4^{+} 
 \end{array}
  $
\end{center}
The third and the fourth cases are similar to the first and second respectively.
\end{proof}

Let us see an example of constructing $Q^{u,v}$. The quiver is obtained by attaching the quiver $Q^{e,v}$ on top of the quiver $Q^{u,e}$. We will see this with $u=s_1s_2s_1s_3$, $v=s_2s_3s_3s_1 \in S_4$ in the following figures. 

\begin{figure}[ht!]
\centering

\begin{tikzpicture}
\node [rectangle,draw,scale=0.5,label=below:$1$,blue] (9) at (3,0) {};
\node [rectangle,draw,scale=0.5,label=below:$$,blue] (10) at (3,1) {};
\node [rectangle,draw,scale=0.5,label=below:$2$,blue] (13) at (4,0) {};
\node [fill,circle,scale=0.5,label=below:$$,blue] (14) at (4,1) {};
\node [rectangle,draw,scale=0.5,label=below:$$,blue] (15) at (4,2) {};

\draw [thick,->,blue] (10)-- (9);
\draw [thick,->,blue] (14)-- (13);
\draw [thick,->,blue] (15)-- (14);

\draw [thick,->, jade] (9)-- (15);
\draw [thick,->, dashed,jade] (13)-- (9);
\draw [thick,->, dashed,jade] (15)-- (10);

\node [rectangle,draw,scale=0.5,label=below:$3$,blue] (26) at (5,0) {};
\node [rectangle,draw,scale=0.5,label=below:$$,blue] (27) at (5,1) {};

\draw [thick,->,blue] (27)-- (26);

\draw [thick,->,red] (14)-- (27);
\draw [thick,->, red] (26)-- (14);
\draw [thick,->, dashed,red] (13)-- (26);
\draw [thick,->, dashed,red] (27)-- (15);

\end{tikzpicture}
\label{top}
\caption{$Q^{e,v}$, $v=s_2s_3s_2s_1$}
 
 \medskip

\begin{tikzpicture}

\node [rectangle,draw,scale=0.5,label=below:$1$] (9) at (3,0) {};
\node [fill,circle,scale=0.5,label=below:$$] (10) at (3,1) {};
\node [rectangle,draw,scale=0.5,label=below:$$] (11) at (3,2) {};
\node [rectangle,draw,scale=0.5,label=below:$2$] (13) at (4,0) {};
\node [rectangle,draw,scale=0.5,label=below:$$] (15) at (4,1) {};

\draw [thick,->] (9)-- (10);
\draw [thick,->] (10)-- (11);
\draw [thick,->] (13)-- (15);

\draw [thick,->, jade] (10)-- (13);
\draw [thick,->, jade] (15)-- (10);
\draw [thick,->, dashed,jade] (13)-- (9);
\draw [thick,->,dashed, jade] (11)-- (15);

\node [rectangle,draw,scale=0.5,label=below:$3$] (26) at (5,0) {};
\node [rectangle,draw,scale=0.5,label=below:$$] (27) at (5,1) {};

\draw [thick,->] (26)-- (27);

\draw [thick,->, red] (15)-- (26);
\draw [thick,->, dashed,red] (26)-- (13);
\draw [thick,->, dashed,red] (27)-- (15);

\end{tikzpicture}
\caption{$Q^{u,e}$, $u=s_1s_2s_1s_3$}

\bigskip

\begin{tikzpicture}
\node [rectangle,draw,scale=0.5,label=below:$1$] (9) at (3,0) {};
\node [fill,circle,scale=0.5,label=below:$$] (10) at (3,1) {};
\node [rectangle,draw,scale=0.5,label=below:$$] (11) at (3,2) {};
\node [rectangle,draw,scale=0.5,label=below:$$,blue] (14) at (3,3) {};
\node [rectangle,draw,scale=0.5,label=below:$2$] (13) at (4,0) {};
\node [rectangle,draw,scale=0.5,label=below:$$] (15) at (4,1) {};
\node [fill,circle,scale=0.5,label=below:$$,blue](16) at (4,2) {};
\node [rectangle,draw,scale=0.5,label=below:$$,blue] (17) at (4,3) {};

\draw [thick,->] (9)-- (10);
\draw [thick,->] (10)-- (11);
\draw [thick,->] (13)-- (15);

\draw [thick,->,blue] (14)-- (11);
\draw [thick,->,blue] (16)-- (15);
\draw [thick,->,blue] (17)-- (16);

\draw [thick,->, jade] (11)-- (17);
\draw [thick,->, dashed,jade] (17)-- (14);

\draw [thick,->, jade] (10)-- (13);
\draw [thick,->, jade] (15)-- (10);
\draw [thick,->, dashed,jade] (13)-- (9);

\node [rectangle,draw,scale=0.5,label=below:$3$] (26) at (5,0) {};
\node [rectangle,draw,scale=0.5,label=below:$$] (27) at (5,1) {};
\node [rectangle,draw,scale=0.5,label=below:$$,blue] (28) at (5,2) {};

\draw [thick,->] (26)-- (27);
\draw [thick,->,blue] (28)-- (27);

\draw [thick,->, red] (15)-- (26);
\draw [thick,->, dashed,red] (26)-- (13);

\draw [thick,->, red] (27)-- (16);
\draw [thick,->,red] (16)-- (28);
\draw [thick,->, dashed,red] (28)-- (17);

\end{tikzpicture}
\label{bottom}
\caption{$Q^{u,e}$ and $Q^{e,v}$ glued at frozen vertices}

\medskip 

\begin{tikzpicture}
\node [rectangle,draw,scale=0.5,label=below:$1$] (9) at (3,0) {};
\node [fill,circle,scale=0.5,label=below:$$] (10) at (3,1) {};
\node [fill,circle,scale=0.5,label=below:$$] (11) at (3,2) {};
\node [rectangle,draw,scale=0.5,label=below:$$] (14) at (3,3) {};
\node [rectangle,draw,scale=0.5,label=below:$2$] (13) at (4,0) {};
\node  [fill,circle,scale=0.5,label=below:$$] (15) at (4,1) {};
\node [fill,circle,scale=0.5,label=below:$$] (16) at (4,2) {};
\node [rectangle,draw,scale=0.5,label=below:$$] (17) at (4,3) {};

\draw [thick,->] (9)-- (10);
\draw [thick,->] (10)-- (11);
\draw [thick,->] (13)-- (15);

\draw [thick,->] (14)-- (11);
\draw [thick,->] (16)-- (15);
\draw [thick,->] (17)-- (16);

\draw [thick,->, jade] (11)-- (17);
\draw [thick,->, dashed,jade] (17)-- (14);

\draw [thick,->, jade] (10)-- (13);
\draw [thick,->, jade] (15)-- (10);
\draw [thick,->, dashed,jade] (13)-- (9);

\node [rectangle,draw,scale=0.5,label=below:$3$] (26) at (5,0) {};
\node [fill,circle,scale=0.5,label=below:$$] (27) at (5,1) {};
\node [rectangle,draw,scale=0.5,label=below:$$] (28) at (5,2) {};

\draw [thick,->] (26)-- (27);
\draw [thick,->] (28)-- (27);

\draw [thick,->, red] (15)-- (26);
\draw [thick,->, dashed] (26)-- (13);

\draw [thick,->, red] (27)-- (16);
\draw [thick,->,red] (16)-- (28);
\draw [thick,->, dashed,red] (28)-- (17);

\end{tikzpicture}
\label{bottom}
\caption{$Q^{u,v}$, $u=s_2s_3s_2s_1$, $v=s_2s_3s_2s_1$}

\end{figure}

\begin{lemma} A BFZ quiver $Q^{u,v}$ is planar in each sheet. 
\end{lemma}
\begin{proof}
Consider the $k$th and the $l$th string of the quiver.  If the strings are not adjacent on a sheet, then we know that there cannot be edges between the vertices of the strings.  If the strings are adjacent, consider the following diagram:
\begin{center}
$\begin{tikzcd}
k \ar{r} & k^{+} \ar{dl}\ar{r} & k^{++} \ar{r} & \cdots \ar{r} & k^{p+}\\ 
l \ar{r}  & l^{+} \ar{r} & l^{++} \ar{r} & \cdots \ar{r} & l^{r+} 
\end{tikzcd}$
\end{center}

Suppose the vertices $l$ and $k^{+}$ are connected.  Then depending on whether $k^{+}< l$ or $l< k^{+}$, there will be the following inequalities:\\
(1) If  $k^{+}< l$, then  $ l<k^{++}<l^{+}$ \\
(2) If  $l< k^{+}$, then $k^{+}< l^{+}<k^{++}$. \\
(3) So in both cases above, $k^+ < l^+$. \\
 We want to show that the vertex $k$ is not connected to $l^{m+}$ for any $m$. Suppose $k$ and $l^{m+}$ are connected.  Then again, there are two cases: \\
(4) If  $l^{m+}<k $, then  $ k<l^{(m+1)+}<k^{+}$ \\
(5) If  $k< l^{m+}$, then $l^{m+}< k^{+}<l^{(m+1)+}$. \\
(6) Combining inequalities in (4) and (5) with $l^+<l^{m+}$ we get, $l^+ <k^+$. \\
As (3) and (6) contradict each other, there cannot be an overlapping edge.  Hence the quiver is planar in each sheet.
\end{proof}

\noindent Recall the following theorem from \cite{k17}:
\begin{lemma} For any Kac--Moody algebra $\mathfrak{g}$ and $(u,e) \in W \times W$, all faces of $Q^{u,e}$ are oriented, where $e$ is the identity element in $W$. 
\end{lemma} 

\begin{proposition} All faces of the quiver $Q^{u,v}$ are oriented.
\end{proposition}
\begin{proof} We know from Theorem~\ref{Theorem:gluing} that the quiver $Q^{u,v}$ can be obtained from gluing $Q^{e,v}$ on top of $Q^{u,e}$.  All faces in the quivers $Q^{e,v}$ and $Q^{u,e}$ are oriented.  We need to show that the gluing of two quivers also gives oriented faces.  
\begin{itemize}
\item Case 1.  Suppose the edges between two identified vertices are directed in the same direction.  In this case, we identify the two edges.  In Figure~2, $e$ is the identified edge, $k$ and $l$ are the identified vertices. The paths $p_1$ and $p_2$ are the parts of the oriented faces that contain the edge $e$ in $Q^{u,e}$ and $Q^{e,v}$ respectively.  As shown in the figure, after gluing the vertices and the edge, the resulting faces in $Q^{u,v}$ are still oriented. 

\begin{figure}[ht!] \label{case1}
\centering
\begin{tikzpicture}
\node [fill,circle,scale=0.5,label=left:$k$] (en) at (0,0) {};
\node [fill,circle,scale=0.5,label=right:$l$] (e1) at (2,0) {};
\draw [thick,red, ->] (en)-- node[above]{$e$} (e1);
\draw [thick,dashed, <<-](en) arc [radius=1, start angle=180, end angle=50] node[right]{$p_2$}arc [radius=1, start angle=50, end angle=0] ;

\node [fill,circle,scale=0.5,label=left:$k$] (k) at (0,-1) {};
\node [fill,circle,scale=0.5,label=right:$l$] (l) at (2,-1) {};
\draw [thick,red, ->] (k)-- node[below]{$e$} (l);
\draw [thick, ->>, dashed](l) arc [radius=1, start angle=0, end angle=-30] node[right]{$p_1$} arc [radius=1, start angle=-30, end angle=-180];

\node [fill,circle,scale=0.5,label=left:$k$] (k) at (5,-0.5) {};
\node [fill,circle,scale=0.5,label=right:$l$] (l) at (7,-0.5) {};
\draw [thick,red, ->] (k)-- node[below]{$e$} (l);
\draw [thick, ->>, dashed](l) arc [radius=1, start angle=0, end angle=-30] node[right]{$p_1$} arc [radius=1, start angle=-30, end angle=-180];
\draw [thick,dashed, <<-](k) arc [radius=1, start angle=180, end angle=50] node[right]{$p_2$}arc [radius=1, start angle=50, end angle=0] ;
\end{tikzpicture}
\caption {Case 1}
\end{figure}
\item Case 2. Suppose the edges between two identified vertices are directed in opposite directions.  In this case, we delete the edges.  In Figure~3, $k$ and $l$ are the identified vertices. As the two red edges between $k$ and $l$ are oppositely oriented, there is no edge in the glued diagram.  This still results into an oriented face $F$ in the quiver $Q^{u,v}$ as shown in the figure. 
\end{itemize}

\begin{figure}[ht!]
\centering
\begin{tikzpicture}
\node [fill,circle,scale=0.5,label=left:$k$] (en) at (0,0) {};
\node [fill,circle,scale=0.5,label=right:$l$] (e1) at (2,0) {};
\draw [thick,red, <-] (en)-- node[above]{$e$} (e1);
\draw [thick,dashed, ->>](en) arc [radius=1, start angle=180, end angle=50] node[right]{$p_2$}arc [radius=1, start angle=50, end angle=0] ;

\node [fill,circle,scale=0.5,label=left:$k$] (k) at (0,-1) {};
\node [fill,circle,scale=0.5,label=right:$l$] (l) at (2,-1) {};
\draw [thick,red, ->] (k)-- node[below]{$e$} (l);
\draw [thick, ->>, dashed](l) arc [radius=1, start angle=0, end angle=-30] node[right]{$p_1$} arc [radius=1, start angle=-30, end angle=-180];

\node [fill,circle,scale=0.5,label=left:$k$] (k) at (5,-0.5) {};
\node [fill,circle,scale=0.5,label=right:$l$] (l) at (7,-0.5) {};

\draw [thick, ->>, dashed](l) arc [radius=1, start angle=0, end angle=-30] node[right]{$p_1$} arc [radius=1, start angle=-30, end angle=-180];
\draw [thick,dashed, ->>](k) arc [radius=1, start angle=180, end angle=50] node[right]{$p_2$}arc [radius=1, start angle=50, end angle=0] ;
\node at(6,-0.5) {$F$};
\end{tikzpicture}
\caption {Case 2}
\label{diag:case2} 
\end{figure}

\end{proof}

\begin{theorem} The faces of the quiver on each sheet share at most one edge with each other. 
\end{theorem}
\begin{proof} Let us consider the following part of the Dynkin diagram.  Suppose the two faces share two edges on the rth string. Let $u \in W$ be in its reduced form, then $u$ has the following form, where each of the spaces do not contain $s_{r-1}, s_r$ or$s_{r+1}$.  \[u= \underline{\hspace{0.3cm}1st\hspace{0.3cm}}s_{r+1} \underline{\hspace{0.3cm}2nd\hspace{0.3cm}} s_r   \underline{\hspace{0.3cm}3rd\hspace{0.3cm}} s_{r-1}  \underline{\hspace{0.3cm}4th\hspace{0.3cm}} s_r  \underline{\hspace{0.3cm}5th\hspace{0.3cm}} s_r  \underline{\hspace{0.3cm}6th\hspace{0.3cm}} s_{r-1}  \underline{\hspace{0.3cm}7th\hspace{0.3cm}} s_{r+1}  \underline{\hspace{0.3cm}8th\hspace{0.3cm}}.\]
As the $5th$ space does not contain $s_{r-1}, s_r$ or $s_{r+1}$, \[ s_r  \underline{\hspace{0.35cm}5th\hspace{0.35cm}} s_r= s_r s_r \underline{\hspace{0.35cm}5th\hspace{0.35cm}}= \underline{\hspace{0.35cm}5th\hspace{0.35cm}} \] which implies that $u$ was not in its reduced form.  Therefore, it is not possible for two such faces to share two edges.

\end{proof}

\begin{ex} Two faces of $Q^{u,v}$ can share two edges only on the spine of the cylinder. For example consider the Dynkin diagram $D_4$. Let $u=s_4s_3s_1s_3s_2s_3s_1s_4$.  Then~Figure \ref{2facesharing} shows the quiver $Q^{u,e}$ where two faces share two edges on the spine. 

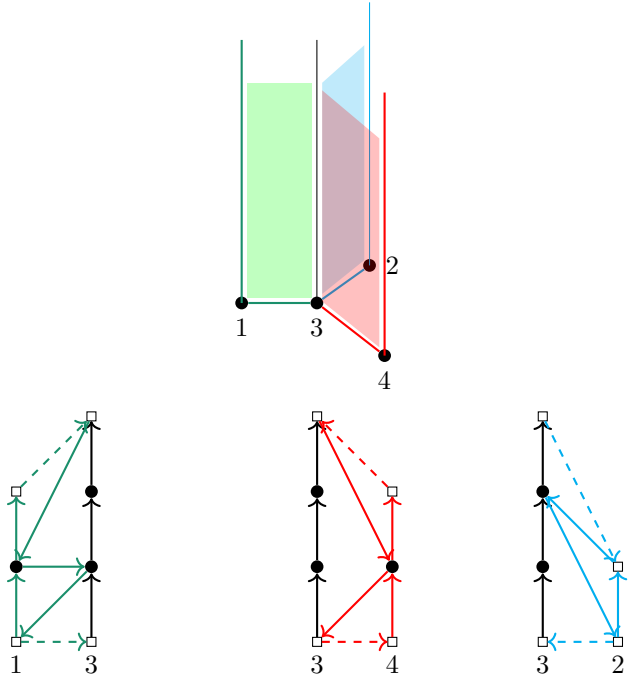
\begin{figure}[ht!]
\centering

\begin{tikzpicture}
\node [fill,circle,scale=0.5,label=below:$1$] (1) at (2,0) {};
\node [fill,circle,scale=0.5,label=below:$4$] (4) at (3.9,-0.7) {};
\node [fill,circle,scale=0.5,label=right:$2$] (3) at (3.7,0.5) {};
\node [fill,circle,scale=0.5,label=below:$3$] (2) at (3,0) {};


\draw [thick,-,cyan] (3)-- (2);
\draw [thick,-,jade] (2)-- (1);
\draw [thick,-,red] (2)-- (4);

\draw  [thick,-,jade] (2,0)-- (2,3.5);
\draw (3,0)-- (3,3.5);
\draw [cyan](3.7,0.5)-- (3.7,4);
\draw [thick,-,red] (3.9,-0.7)-- (3.9,2.8);

\fill[nearly transparent,green]  (2.93,2.92) rectangle (2.07,0.07);

\fill[nearly transparent,cyan] (3.07,0.12) -- (3.07,2.93)--(3.63,3.43)--(3.63,0.57);
\fill[nearly transparent,red] (3.07,0.02) -- (3.07,2.83)--(3.83,2.19)--(3.83,-0.59);
\end{tikzpicture}

\medskip
\centering

\begin{tikzpicture}
\node [rectangle,draw,scale=0.5,label=below:$1$] (9) at (2,0) {};
\node [fill,circle,scale=0.5,label=below:$$] (10) at (2,1) {};
\node [rectangle,draw,scale=0.5,label=below:$$] (11) at (2,2) {};
\node [rectangle,draw,scale=0.5,label=below:$3$] (13) at (3,0) {};
\node [rectangle,draw,scale=0.5,label=below:$$] (15) at (3,3) {};
\node [fill,circle,scale=0.5,label=below:$$] (16) at (3,2) {};
\node [fill,circle,scale=0.5,label=below:$$] (17) at (3,1) {};

\draw [thick,->, jade] (9)-- (10);
\draw [thick,->, jade] (10)-- (11);
\draw [thick,->] (13)-- (17);
\draw [thick,->] (16)-- (15);
\draw [thick,->] (17)-- (16);

\draw [thick,->, jade] (17)-- (9);
\draw [thick,->, jade] (10)-- (17);
\draw [thick,->, jade] (15)-- (10);
\draw [thick,->, dashed,jade] (9)-- (13);
\draw [thick,->,dashed, jade] (11)-- (15);

\node [rectangle,draw,scale=0.5,label=below:$3$] (22) at (6,0) {};
\node [fill,circle,scale=0.5,label=below:$$] (23) at (6,1) {};
\node [fill,circle,scale=0.5,label=below:$$] (24) at (6,2) {};
\node [rectangle,draw,scale=0.5,label=below:$$] (25) at (6,3) {};
\node [rectangle,draw,scale=0.5,label=below:$4$] (26) at (7,0) {};
\node [fill,circle,scale=0.5,label=below:$$] (27) at (7,1) {};
\node [rectangle,draw,scale=0.5,label=below:$$] (28) at (7,2) {};

\draw [thick,->] (23)-- (24);
\draw [thick,->] (22)-- (23);
\draw [thick,->] (24)-- (25);
\draw [thick,->,red] (27)-- (28);
\draw [thick,->, red] (26)-- (27);

\draw [thick,->, red] (25)-- (27);
\draw [thick,->, red] (27)-- (22);
\draw [thick,->, dashed,red] (22)-- (26);
\draw [thick,->, dashed,red] (28)-- (25);

\node [rectangle,draw,scale=0.5,label=below:$3$] (30) at (9,0) {};
\node [fill,circle,scale=0.5,label=below:$$] (31) at (9,1) {};
\node [fill,circle,scale=0.5,label=below:$$] (32) at (9,2) {};
\node [rectangle,draw,scale=0.5,label=below:$$] (33) at (9,3) {};
\node [rectangle,draw,scale=0.5,label=below:$2$] (34) at (10,0) {};
\node [rectangle,draw,scale=0.5,label=below:$$] (35) at (10,1) {};

\draw [thick,->] (32)-- (33);
\draw [thick,->] (31)-- (32);
\draw [thick,->] (30)-- (31);
\draw [thick,->, cyan] (34)-- (35);

\draw [thick,->, cyan] (32)-- (34);
\draw [thick,->, cyan] (35)-- (32);
\draw [thick,->, dashed,cyan] (34)-- (30);
\draw [thick,->, dashed,cyan] (33)-- (35);

\end{tikzpicture}
\caption{The quiver $Q^{u,e}$ for $u=s_4s_3s_1s_3s_2s_3s_1s_4$ in $D_4$ showing that the top two edges on String 3 are shared by two faces of the quiver.}
\label{2facesharing}
\end{figure}
\end{ex}

\begin{defn} We say that a quiver can be drawn on the cylinder $\Gamma \times \RR^{\geq0}$ if  the vertices of the quiver are on the grid $\Gamma_0 \times \NN$ and edges lie on the sheets of the cylinder. 
\end{defn}
\begin{enumerate}
\item Draw a vertical line on each vertex of the graph. We will call them strings
\item Place jth vertex of the BFZ quiver over the vertex $i_j$ of the Dynkin diagram at height $j$.
\item Connect the vertices with arrows as in BFZ
	\item All vertices of the quiver lie on the strings. Each adjacent pair of vertices on a string is connected by vertical edges of the quiver. The vertical egdes and vertices project down on the corresponding vertex of the graph.
	\item Other inclined edges between the vertices of the quiver project down on the edges of the graph.
\end{enumerate}

\begin{remark} In $Q^{u,v}$, a face with $n$ vertices has $n-1$ vertices on one string and the remaining one vertex on its adjacent string. The second diagram in~Figure \ref{nfaces} shows a cycle that does not appear in $Q^{u,v}$.
\begin{figure}
\begin{center}
\begin{tikzcd}
j \ar{r} & j^{+} \ar{r} & j^{++} \ar{r} & \cdots \ar{r} & j^{(n-2)+} \ar{dll}\\ 
 &  & k \ar{ull} 
\end{tikzcd}

\begin{tikzcd}
j \ar{r} & j^{+} \ar{r} & j^{++} \ar{r} & \cdots \ar{r} & j^{p+} \ar{dl}\\ 
 &  & k  \ar{ull} & k^{+}  \ar{l} 
\end{tikzcd}
\end{center} 
\caption{Types of faces in a dimer model}
\label{nfaces}
\end{figure}
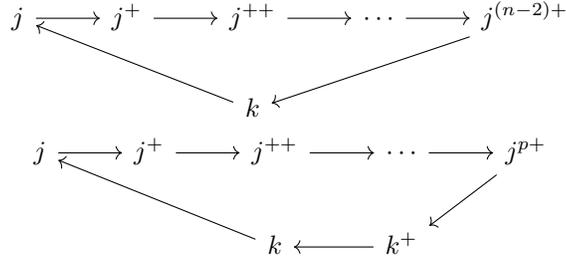

\end{remark} 

\section{Rigidity of the superpotential}
In this section, we will show that the superpotential of $Q^{u.v}$ is rigid.  The potential  \[S= \sum{\text{clockwise oriented faces}}-\sum{\text{anti-clockwise oriented faces}}\] is called the superpotential of the quiver $Q$.  From \cite{birs09} and \cite{k17}, we know that the superpotential of $Q^{u,e}$ is rigid for all symmetric Kac-Moody algebras.  Now we will use construction of $Q^{u,v}$ in Theorem~3.2 to prove the rigidity in general. 

The technique of drawing quivers on the cylinders allows us to look at the quiver in its planar parts.  Consider the subquiver of $Q^{u,v}$ consisting of all vertices and edges lying on a sheet of the cylinder.  The subquiver on each sheet of the cylinder has a superpotential.  The superpotential of $Q^{u,v}$ is just the sum of all such smaller superpotentials on each sheet. 

Let us recall some definitions.  A potential $S \in \CC(Q)$  is a linear combination of cycles in the quiver. The pair $(Q,S)$ of a quiver and its potential is called a quiver with potential or a QP.  Note that all cycles in a potential are simple.  This means that no cycle passes through the same vertex twice.  We will follow the definition of  mutation of quivers with potential in \cite{dwz08}. 
 \begin{itemize}\item Two potentials $S$ and $S'$ on Q are cyclically equivalent if $S-S'$ lies in the closure of the vector subspace spanned by all the elements of the form $\alpha_1 \ldots \alpha_l - \alpha_2 \ldots \alpha_l \alpha_1$ where $\alpha_1 \ldots \alpha_l $ is a cycle of positive length.  
 \item Two quivers with potentials $(Q,S)$ and $(Q',S')$ are right equivalent to each other if there exists an isomorphism $\phi: \CC(Q) \to \CC(Q')$ such that $\phi(S)$ is cyclically equivalent to $S'$.
 \item Let us define the cyclic derivative of a potential.  For every $a \in Q_1$, the cyclic derivative $\partial_a$ is defined as: 
\[ \partial_a(a_1a_2\cdots a_n)= \sum_{i} \delta_{a a_i}a_{i+1}\cdots a_na_1\cdots a_{i-1},\]  where $a_1a_2\cdots a_n$ is a cycle in the quiver and  $a=a_i$.  If $a \neq a_i$ for any $i$, then $\partial_a(a_1a_2\cdots a_n)=  0$.
\item If $S$ is a potential of $Q$, we define the Jacobian ideal $J(S)$ to be the ideal generated by $\partial_a(S)$, for all $a \in Q$. 
\item The Jacobian algebra $P(Q,S)$ is the quotient $\CC(Q)/J(S)$. 
 \item A QP is called trivial if it is a sum of cycles of length 2, and the derivatives span $Q$ as a $\CC$-vector space. 
 \item A QP is called reduced if the degree-2 component of $S$ is 0, i.e., if the expression of $S$ involves no 2-cycles.
 \end{itemize}

\begin{defn} A QP $(Q,S)$ is rigid if every cycle in Q is cyclically equivalent to an element of the Jacobian ideal $J(S)$.
\end{defn}

In order to show that $S$ is rigid in $Q^{u,v}$, we use the following result from \cite{k17}: 
\begin{theorem}The superpotential $S$ of the quiver $Q^{u,e}$ is a rigid potential. 
\end{theorem}
The proof in \cite{k17} in the case of $Q^{u,e}$ used the fact that each face is oriented and the fact that we are allowed to use cycles up to cyclic equivalence.  First  we proved that every face belongs to $J(S_r)$, where $S_r$ is the superpotential of the quiver on the rth sheet. The proof of this is by induction on the distance of a face from the boundary of the quiver.  Once we show that each face belongs to $J(S_r)$, we observe that every other cycle in the quiver is in fact multiplication of a face and a cycle in the quiver.  Since every face belongs to the Jacobian ideal, the original cycle also does. 

This argument can still be applied to the quivers $Q^{u,v}$ due to their structure observed in Theorem~\ref{Theorem:gluing}.  We proved here that each face of the $Q^{u,v}$ is also oriented and any two faces on a sheet share at most one edge.  Hence the induction argument from \cite{k17} applies.  In particular, the above result is also true for the superpotential of a quiver $Q^{e,v}$.  This is because the identity element $e$ does not contribute anything to the quiver.  It only decides whether the indices for the word $v$ are positive or negative (which determines the direction of the vertical arrows).  So the quivers $Q^{e,v}$ and $Q^{v,e}$ are isomorphic.  
Following Theorem \ref{Theorem:gluing}, we see that $Q^{u,v}$ can be obtained by gluing $Q^{u,e}$ and $Q^{e,v}$.  Since each cycle in the resulting quiver $Q^{u,v}$ is oriented, the superpotential of the quiver contains all its faces. Also the superpotential of each of the quivers  $Q^{u,e}$ and $Q^{e,v}$ is rigid.  Hence after gluing the quivers, the superpotential on each sheet is rigid. 
\begin{defn} A cycle $C$ will be called differentiable with respect to an edge $e$ if there exists an edge $e$ in the interior of the cycle that separates the cycle into a face and a smaller cycle.
\end{defn}
It is proved in \cite{k17} that every cycle in the quiver  $Q^{u,e}$ is differentiable. A similar argument proves that any cycle in the quiver  $Q^{u,v}$ is also differentiable.  The core of the argument is that taking derivative of the cycle with respect to an edge pushes the cycle in its interior.  This process can be iterated until there are two vertices in the cycle that are connected by an edge which is not in the cycle.  
\begin{corollary}The superpotential $S$ of the quiver $Q^{u,v}$ is a rigid potential. 
\end{corollary}
\begin{proof}
If a cycle lies on a single sheet of the cylinder, it is in the Jacobian ideal generated by derivatives of the superpotential on that sheet.  Hence it is also in $J(S)$.  If the cycle $C$ lies in more than one sheets, we know that it is differentiable with respect to an edge, say $e: v_1 \to v_2$.  We know that in the Jacobian algebra, two paths from $v_1$ to $v_2$ are on the same sheet and are identified. One of these paths is a part of $C$ and the other is in the interior of $C$. Let us replace the path that is on $C$ by its equivalent path in the interior of $C$.  In other words, this reduces the area covered by $C$.  We continue this process until $C$ can be written as a product of a face and a cycle.  Since a face always belongs to exactly one sheet, and it is already in the Jacobian ideal $J(S)$, so are any of its multiples.  This proves that $C$ is also belongs to the Jacobian ideal. 
\end{proof}

%
%

\section{Future interests}

The quivers from double Bruhat cells have many frozen vertices and hence, in order to categorify these cluster algebras with frozen variables, one has to find an appropriate model similar to Pressland's model in \cite{p17}.  The methods in Pressland's work cannot be applied directly as the number of frozen vertices do not exactly match in these two cases.  The quivers defined in \cite{p17} have the number of frozen vertices double the number of mutable vertices, whereas the number of frozen and mutable in our case do not have this relation.  This is one of the reasons we add the arrows in the quiver between frozen vertices.  The quivers in Pressland's setting have arrows connecting frozen vertices following a certain rule.  This rule seems to match the current rule we have in our setting.  The cycles obtained in this way are also oriented. 

The frozen Jacobian algebras can be defined for the BFZ quivers with superpotentials, as defined in \cite{p17}.  This algebra does not include derivatives with respect to boundary arrows. In the setting of this thesis, we need to include the cyclic derivatives with respect to the edges between frozen vertices (or the boundary arrows), i.e., for a quiver $Q$, a superpotential $W$, and a frozen subquiver $F$,  \[A=J(Q,F,W) = \CC(Q)/\langle\partial_aW | a \in Q_1\rangle.\]
The boundary algebra $B$ will be $eAe$ where the idempotent element $e$ is the sum of idempotents at every frozen vertex.  In order to apply Pressland's result, the first step would be to show that the frozen Jacobian algebra is bimodule internally 3-Calabi--Yau.  For a cluster tilting object whose endomorphism algebra is isomorphic to the Jacobian algebra $A(Q,W)$, the mutation of the titling object is not always isomorphic to the Jacobian algebra of a mutation of the potential.  In \cite{birs11}, the necessary conditions for the above statement to hold are given.  We would like to study those conditions in this case of BFZ quivers and the superpotentials.

Moreover, the goal of this project is not just to categorify the cluster structure, but also to recognize the cluster monomials as the elements of the Lusztig's dual canonical basis.  Geiss, Leclerc and Schr\"{o}er were able to do this using their categorification in type ADE.  Even though the result is not exactly the original motivation for the definition of cluster algebras, it brings us closer to it since cluster monomials are shown to be inside the dual of the required basis.

\bibliographystyle{alpha}

\end{document}